\documentclass[12pt, a4paper]{amsart}
\usepackage{amsmath}
\usepackage{geometry,amsthm,graphics,tabularx,amssymb,
shapepar}
\usepackage{amscd}
\usepackage[usenames]{color}

\usepackage{graphicx}

\usepackage[all]{xy}

\newcommand{\be}{\begin {equation}}
\newcommand{\ee}{\end {equation}}
\newcommand{\bee}{\begin {equation*}}
\newcommand{\eee}{\end {equation*}}

\theoremstyle{Theorem}

\theoremstyle{Theorem}

\theoremstyle{Theorem}
\newtheorem{prp}{Proposition}[section]

\newtheorem{lemp}[prp]{Lemma}
\newtheorem{thmp}[prp]{Theorem}

\theoremstyle{Plain}

\theoremstyle{Definition}

\begin{document}

\title[Normalized Laplacian Spectrum]{The normalized Laplacian spectrum and eigentime identities of hype-cubes}

\author[Y. Y. Chen]{Yangyang Chen}

\address{School of Sciences, Harbin Institute of Technology, Shenzhen, 518055, China}
\email{chenyangyang@hit.edu.cn}

\author[Y. Zhao]{Yi Zhao}
\address{School of Sciences, Harbin Institute of Technology, Shenzhen, 518055, China}
\email{zhao.yi@hit.edu.cn}



\begin{abstract}
Many popular graph metrics encode average properties of individual network elements. Complementing these conventional graph metrics, the eigenvalue spectrum of the normalized Laplacian describes a network's structure directly at a systems level, without referring to individual nodes or connections. In this paper, we study 
the spectrum and their applications of normalized Laplacian matrices of hype-cubes, a special kind of Cayley graphs. We determine explicitly all the eigenvalues and their
corresponding multiplicities by a recursive method. By using the relation between normalized Laplacian spectrum and eigentime identity, we derive the explicit formula to the eigentime
identity for random walks on the hype-cubes and show that it grows linearly with the network size. Moreover, we compute the number of spanning trees of the hype-cubes.
\end{abstract}

 \maketitle

\section{Introduction}
Recently, the theory of complex networks has attracted wide attention and becomes an area of great interest \cite{AB,Ne}, for its advances in the understanding of many natural and 
social systems. A central issue in the study of complex systems is to understand the topological structure and to further unveil how various structural properties affect the 
dynamical processes occurring on diverse systems \cite{GB}. From a graph-theoretic perspective, the spectrum of the standard Laplacian matrix of a network contains tremendous information 
about the underlying network, which provides useful insights into the intrinsic structural features of the network \cite{CD} and plays a fundamental role in the dynamical behavior of the network. 
For example, the resistance distance \cite{Wu}, relaxation dynamic in the framework of generalized Gaussian structure \cite{GB,SB,Sc}, fluorescence depolarization by quasiresonant
energy transfer \cite{BVJK,BVJK1,LZ}, continuous-time quantum walks \cite{ADZ,ABM,MB}, average trapping time \cite{DFW} and so on. Thus it is important to study the spectrum of standard Laplacian matrices of complex networks. For the techniques to compute the spectrum of standard Laplacian matrices of complex systems, we refer the reader to earlier works \cite{GB,BJ,JVB,Ga}.

Compared to standard Laplacian matrices, the spectrum of normalized Laplacian matrices have received little attention \cite{CLV,KK,WZ}. However, the eigenvalues and eigenvectors
of normalized Laplacian matrix of a network also contain much important information about its structure and dynamical processes. For example, the number of spanning 
tress of a connected network is determined by the product of all nonzero eigenvalues\cite{Chu}; the nonzero eigenvalues and their orthonormalized eigenvectors can be used to
describe the resistor resistance between any pair of nodes \cite{CZ}. Many interesting quantities of random walks, like mixing time \cite{AF}, Kemeny constant \cite{KS} and 
eigentime identity \cite{XY}, are related to the normalized spectrum. Moreover, many problems in chemical physics \cite{BKK,BZ} are closely related to eigenvalues and eigenvectors of normalized Laplacian matrix.

In \cite{JWZ}, Julaiti et al. studied the spectrum of normalized Laplacian matrices of a family of fractal trees and dendrimers modeled by Cayley trees, both of which were built in an iterative way. They used recursive method to determine all the eigenvalues and their corresponding multiplicities. As an application, they obtained an explicit solution to the eigentime identity for random walks. Eigentime identity for random walks were also studied in \cite{DWS},
for a family of treelike networks and polymer networks.

It should be pointed out that \cite{JWZ} did not give an explicit formula of the normalized spectrum. Instead, they gave a recursive relation governing the eigenvalues of fractal trees
at two successive generations and for the Cayley tress, the eigenvalues were described as roots of several small-degree polynomials defined recursively. Generally, it should be a difficult problem to determine explicitly the normalized Laplacian spectrum of networks. However, we believe that for vertex symmetric networks \cite{AK}, the study of normalized Laplacian spectrum will be relatively easy, due to the symmetry of the network structure. Recall that for vertex symmetric networks, there exists a good model, the celebrated Cayley graph model \cite{AK}.  The Cayley graph model has a simple mathematical characterization. It should be an interesting and fascinating problem to study the normalized Laplacian spectrum of Cayley graph networks.

In this paper, we study the normalized Laplacian spectrum of a special kind of Cayley graphs, the hype-cube or $n$-cube, which is a network of $2^n$ vertices, with degree $n$ and diameter $n$. We determine explicitly all the characteristic polynomials of the normalized Laplacian matrices of the $n$-cubes by a recursive method. Particularly, these polynomials are factorized into products of monomials and the roots of these polynomials are elegantly distributed between the closed interval from $0$ to $2$. As an application of these results, we obtain directly the eigentime identity for random walks on these $n$-cubes and the number of spanning trees.

The rest of this paper is organized as follows. In the next section, we recall briefly the Cayley graph models and a special Cayley graph, the hype-cubes or $n$-cubes.
Section \ref{nls} contains the main result of this paper, where we compute the eigenvalues of normalized Laplacian matrix of
hype-cubes. Then in Section \ref{app}, we give two applications of the normalized Laplacian spectrum, namely, the explicit formulas for eigentime identities for random walks and number of spanning trees. Finally, the last section contains our conclusions.

\section{Cayley graph models and hype-cubes}

A graph or network is denoted by $\Gamma = \Gamma(V,E)$, where $V$ is the set of vertices and $E\subset V\times V$ is the set of edges. We only consider graphs that are
finite, undirected, loop-free and devoid of multiple edges in this paper. If $(v_1,v_2) \in E$, $v_1$ and $v_2$ are adjacent. Recall that a network is said to be vertex symmetric, if for any two vertices $v$ and $w$, there exists an automorphism of the network that maps $v$ into $w$. Vertex symmetric networks have the property that the network viewed from any vertex of
the network looks the same. In such a network, congestion problems are minimized since the load will be distributed through all the vertices. It is well-known that Cayley graph model is an excellent model for vertex symmetric networks. It was shown in \cite{AK} that most vertex symmetric networks can be represented using this model, and that every vertex symmetric network can be represented by a simple extension of this model \cite[Theorem 3]{AK}. 

We recall the construction of this model briefly. Let $G$ be a finite group, with a generating subset $S$, namely, all group elements of $G$ can be expressed as a finite product of the powers of the elements in $S$. The Cayley graph of the group $G$ with respect to the subset $S$, denoted by $\mathrm{Cay}(G,S)$, has vertices that are elements of $G$ and edges that are ordered pairs $(g,gs)$ for $g \in G, s \in S$. We always require that $e\not\in S$ and $S = S^{-1}$, where $e$ is the identity element of $G$. Then $\mathrm{Cay}(G,S)$ can be taken as a simple undirected graph. For more definitions and basic results on graphs and groups we refer the reader to \cite{Big}. 

For example, let $G_n = \mathbb{Z}_{2}^{n}$ and $S_n = \{(x_1,...,x_n)\in G_n : \text{only one $x_i$ is 1}\}$, where $\mathbb{Z}_{2} = \mathbb{Z}/2\mathbb{Z}$ denotes the group with only two elements and $\mathbb{Z}_{2}^{n}$ denotes the $n$th direct product of $\mathbb{Z}_{2}$. Then $\Gamma_n = \mathrm{Cay}(G_n,S_n)$ is the well-known hype-cube or $n$-cube, which is a network of $2^n$ vertices, with degree $n$ and diameter $n$ \cite{AK}.

\section{Normalized Laplacian spectrum of hype-cubes}\label{nls}

\subsection{Normalized Laplacian matrix}

Let $\Gamma$ be a network. Denote by $A$ its adjacency matrix, the entry $A(i,j)$ of which is 1 (or 0) if nodes $i$ and $j$ are (not) adjacent in $\Gamma$. Then the standard Laplacian matrix of $\Gamma$ is defined as $L = D - A$, where $D$ is the diagonal degree matrix of $\Gamma$ with its $i$th diagonal entry being the degree of node $i$ in $\Gamma$. Since $L$ is real symmetric, all its eigenvalues are real numbers. Actually, $L$ is positive semi-definite and thus has nonnegative eigenvalues. Moreover, 0 is always an eigenvalue of $L$ and the multiplicity of the eigenvalue 0 is equal to the number of connected components of $\Gamma$. For these facts, see, for example \cite{Hal}. In this paper, we only consider connected networks. 

The normalized Laplacian matrix of $\Gamma$ is defined as $\mathcal{L} = I - D^{-1/2}AD^{-1/2}$, where $I$ denotes the identity matrix with the same order as that of $A$. 
If $a$ is a eigenvalue of $\mathcal{L}$, then $0\leq a \leq 2$, see for example \cite{Chu}. By the normalized Laplacian spectrum of a network, we mean all the eigenvalues of the normalized Laplacian matrix. Recently, it was pointed out by \cite{JWZ} that one has to treat the
standard Laplacian matrix and normalized Laplacian matrix separately, since they behave quite differently. 
The main goal of this paper is to determine the normalized Laplacian spectrum of hype-cubes.

\subsection{Normalized Laplacian spectrum of hype-cubes}
$ $

 Recall that $\Gamma_{n} = \mathrm{Cay}(G_n,S_n)$ denote the $n$-cube, where
$G_n = \mathbb{Z}_2^{n}$ and $S_n = \{(x_1,...,x_n)\in G_n : \text{only one $x_i$ is 1}\}$. We array the vertices in $\Gamma_n$ in lexicographical order. Denote by $A_n$ the corresponding
adjacency matrix and $D_n$ the diagonal degree matrix of $\Gamma_n$. Obviously, $D_n = nI_{2^n}$, where $I_{2^n}$ denotes the identity matrix of order $2^n$. The normalized Laplacian
matrix of $\Gamma_n$ is $\mathcal{L}_{n} = I_{2^n} - D_n^{-1/2}A_nD_n^{-1/2}$, which is equal to $I_{2^n} - D_n^{-1}A_n$, since the degree matrix $D_n$ is a scalar matrix.
 Let $g_n(\lambda) = \mathrm{det}(\lambda I - \mathcal{L}_n) = \mathrm{det}((\lambda - 1)I + D^{-1}A)$ be the characteristic polynomial of the normalized Laplacian matrix $\mathcal{L}_{n}$ of 
 $\Gamma_n$. We sometimes omit the subscripts if it causes no confusions. The main goal of this subsection is to find all the roots of $g_{n}(\lambda)$.
 
 Let $f_{n}(\lambda) = \mathrm{det}((\lambda -1)D_n + A_n)$. It is clear that $g_n(\lambda) = n^{-2^n}f_n(\lambda)$. Thus it suffices to find all the roots of $f_n(\lambda)$. Denote by $A_{n-1}$ the corresponding adjacency matrix of the $(n -1)$-cube $\Gamma_{n-1}$. The following observation makes it possible for us to derive all the roots of $g_{n}(\lambda)$ in a recursive way.
 
 \begin{prp}\label{blo}
 For every $n \geq 2$, we have 
 \[
 A_n = \left(
 \begin{matrix}
 A_{n-1} & I\\
 I & A_{n-1}
 \end{matrix}
 \right).
 \]
\end{prp}

\begin{proof}
This follows directly from the lexicographical order of the vertices of the $n$-cube.
\end{proof}

By Proposition \ref{blo}, we have
\be\label{block}
(\lambda - 1)D_n + A_n = \left(
 \begin{matrix}
 (\lambda -1)nI + A_{n-1} & I\\
 I &  (\lambda -1)nI + A_{n-1}
 \end{matrix}
 \right).
\ee
We shall need the following elementary lemma from linear algebra.

\begin{lemp}\label{det}
Let 
\[
B = \left(
 \begin{matrix}
 A & I\\
 I & A
 \end{matrix}
 \right),
\]
where $A$ is a matrix of order $m$ and $I$ is the identity matrix of the same order. Then 
\[
\mathrm{det}(B) = \mathrm{det}(A + I)\mathrm{det}(A - I).
\]
\end{lemp}

\begin{proof}
By adding the second row of $B$ to the first row and then subtracting the first column from the second column, we have
\[
\mathrm{det}(B) = \mathrm{det}\left(
 \begin{matrix}
 A+I& 0\\
 I & A-I
 \end{matrix}
 \right) =\mathrm{det}(A+I)\mathrm{det}(A-I).
\]
\end{proof}

Recall that $f_{n}(\lambda) = \mathrm{det}((\lambda -1)D_n + A_n)$. By the block decomposition \eqref{block} and Lemma \ref{det}, we have the following recursive relation between 
$f_{n}(\lambda)$ and $f_{n-1}(\lambda)$.

\begin{prp}\label{recur}
For every $n \geq 2$, we have
\be\label{re}
f_{n}(\lambda) = f_{n-1}(\frac{n}{n-1}\lambda)f_{n-1}(\frac{n\lambda - 2}{n-1}).
\ee
\end{prp}

\begin{proof}
By definition, 
\begin{eqnarray*}
f_{n}(\lambda) &=& \mathrm{det}((\lambda -1)D_n + A_n) \\
&=& \mathrm{det}((n(\lambda -1)+1)I + A_{n-1})\mathrm{det}((n(\lambda -1)-1)I + A_{n-1}),
\end{eqnarray*}
the right hand side of the above equation is easily checked to be 
\[
f_{n-1}(\frac{n}{n-1}\lambda)f_{n-1}(\frac{n\lambda - 2}{n-1}),
\]
finishing the proof of the proposition.
\end{proof}

It follows directly that $f_{1}(\lambda) = \lambda(\lambda - 2)$. Using the recursive relation \eqref{re}, we can determine $f_{n}(\lambda)$ explicitly.

\begin{thmp}\label{res}
For $n \geq 1$,
\be\label{main}
f_{n}(\lambda) = n^{2^n}\prod_{k=0}^{n}(\lambda - \frac{2k}{n})^{\binom{n}{k}}.
\ee
\end{thmp}

\begin{proof}
We prove \eqref{main} by induction on $n$. Obviously the case $n = 1$ holds. Assume that \eqref{main} holds for $n -1$, where $n > 1$. We shall show that then it also holds for $n$.
By \eqref{re} and the induction hypothesis on $n -1$, we have
\begin{eqnarray*}
f_n(\lambda) &=& f_{n-1}(\frac{n}{n-1}\lambda)f_{n-1}(\frac{n\lambda - 2}{n-1})\\
&=& (n-1)^{2^{n-1}}\prod_{k=0}^{n-1}(\frac{n}{n-1}\lambda - \frac{2k}{n-1})^{\binom{n-1}{k}}(n-1)^{2^{n-1}}\prod_{k=0}^{n-1}(\frac{n\lambda - 2}{n-1} - \frac{2k}{n-1})^{\binom{n-1}{k}} \\
&=& (n-1)^{2^n}(\frac{n}{n-1})^{2\sum_{k=0}^{n-1}\binom{n-1}{k}}\prod_{k=0}^{n-1}(\lambda - \frac{2k}{n})^{\binom{n-1}{k}}(\lambda - \frac{2k+2}{n})^{\binom{n-1}{k}}\\
&=& n^{2^n}\prod_{k=0}^{n-1}(\lambda - \frac{2k}{n})^{\binom{n-1}{k}}\prod_{k=1}^{n}(\lambda - \frac{2k}{n})^{\binom{n-1}{k-1}} \\
&=& n^{2^n}\prod_{k=0}^{n}(\lambda - \frac{2k}{n})^{\binom{n}{k}},
\end{eqnarray*}
which completes the proof of $n$ case.
\end{proof}

Recall that $g_n(\lambda) = \mathrm{det}(\lambda I - \mathcal{L}_n)$, where $\mathcal{L}_n$ denotes the normalized Laplacian matrix of the $n$-cube $\Gamma_n$ and we have
$g_n(\lambda) = n^{-2^n}f_n(\lambda)$. By Theorem \ref{res}, 
\be\label{poly}
g_n(\lambda) = \prod_{k=0}^{n}(\lambda - \frac{2k}{n})^{\binom{n}{k}}. 
\ee
From this formula, we know all the eigenvalues of normalized Laplacian matrix of the $n$-cube. The eigenvalues are $2k/n$, with multiplicity $\binom{n}{k}$, for each $0\leq k\leq n$.
These eigenvalues are evenly distributed in the closed interval from 0 to 2. As we mentioned in the Introduction, this is partly due to the symmetry of the network structure. We believe that similar nice results hold for other vertex symmetric networks, and specially for Cayley graph networks.

\section{Applications of normalized Laplacian spectrum}\label{app}

As described in the Introduction, the normalized Laplacian spectrum of a network contains much important information about its structure and dynamical processes.
With the normalized Laplacian spectrum of hype-cubes obtained, now we can give explicit formulas to the eigentime identity for random walks on the $n$-cube $\Gamma_{n}$ 
and the number of spanning trees of $\Gamma_{n}$.

\subsection{Eigentime identity for random walks}

Firstly we recall the eigentime identity for random walks in a general network $\Gamma$. Let $H_{ij}$ be the mean-first passage time from node $i$ to node $j$ in $\Gamma$, which is the expected time for a particle starting off from node $i$ to arrive at node $j$ for the first time, see \cite{JWZ}. The stationary distribution for random walks on 
$\Gamma$ is $\pi = (\pi_1, ..., \pi_N)$, where $\pi_i = d_i/2|E(\Gamma)|$, $N = |V(\Gamma)|$ and $d_i$ denotes the degree of node $i$. Let $H$ represent the eigentime identity for random walks in $\Gamma$, which is defined as the expected time for a walker going from a node $i$ to another node $j$, chosen randomly from all nodes accordingly to the stationary distribution. That is,
\be
H = \sum_{j=1}^{N}\pi_j H_{ij}.
\ee
Note that $H$ is independent of the starting node. It is a global characteristic of the network and reflects the architecture of the whole network. By \cite{AF,LL}, $H$ can be expressed as
\be\label{eig}
H = \sum_{\lambda\not = 0}\frac{1}{\lambda},
\ee
where the sum is taken over all the nonzero eigenvalues of the normalized Laplacian matrix of $\Gamma$. 
For recent work on eigentime identities of flower networks with multiple branches and weighted scale-free triangulation networks, we refer the reader to \cite{XYY,DLC}.
We shall give an explicit formula of \eqref{eig} for hype-cubes.

\begin{prp}\label{eigentime}
The eigentime identity for random walks on the $n$-cube is given by
\be
H(\Gamma_n) = \sum_{k=1}^{n}\frac{n\binom{n}{k}}{2k}.
\ee 
\end{prp}

\begin{proof}
This follows directly from Theorem \ref{res}.
\end{proof}

For the asymptotical behavior of the eigentime identity for random walks on $\Gamma_n$ as $n\to\infty$, we have the following characterization.

\begin{prp}\label{asy}
\[
\lim\limits_{n\to\infty}\frac{H(\Gamma_n)}{2^n} = 1.
\]
\end{prp}

\begin{proof}
We have 
$$H(\Gamma_n) = \sum_{k=1}^{n}\frac{n\binom{n}{k}}{2k} = \frac{n}{2}\sum_{k=1}^{n}\frac{1}{k(k+1)}\binom{n}{k} + \frac{n}{2}\sum_{k=1}^{n}\frac{1}{k+1}\binom{n}{k}.$$
Denote by 
\[
S(n) = \frac{n}{2}\sum_{k=1}^{n}\frac{1}{k+1}\binom{n}{k},\quad L(n) = \frac{n}{2}\sum_{k=1}^{n}\frac{1}{k(k+1)}\binom{n}{k}
\]
and
\[
T(n) = \frac{n}{2}\sum_{k=1}^{n}\frac{1}{(k+1)(k+2)}\binom{n}{k}.
\]
Then 
\be
S(n) = \frac{n}{2(n+1)}\sum_{k=1}^{n}\binom{n+1}{k+1} = \frac{n(2^{n+1}-n-2)}{2(n+1)},
\ee
thus
\be\label{mai}
\lim\limits_{n\to\infty}\frac{S(n)}{2^n} = 1.
\ee
On the other hand, 
\[
T(n) = \frac{n}{2(n+1)(n+2)}\sum_{k=1}^{n}\binom{n+2}{k+2} < \frac{n2^{n+1}}{(n+1)(n+2)},
\]
thus
\[
\lim\limits_{n\to\infty}\frac{T(n)}{2^n} = 0.
\]
For any $k\geq 1$, one has
\[
\frac{1}{(k+1)(k+2)} < \frac{1}{k(k+1)} \leq \frac{3}{(k+1)(k+2)},
\]
thus $T(n) < L(n) \leq 3T(n)$ and therefore 
\be\label{eps}
\lim\limits_{n\to\infty}\frac{L(n)}{2^n} = 0.
\ee
Note that $H(\Gamma_n) = L(n) + S(n)$, thus the proposition follows from \eqref{mai} and \eqref{eps}.
\end{proof}

By Proposition \ref{asy}, $H(\Gamma_n)$ grows linearly with the network size $N(\Gamma_n)$ of the $n$-cube as $n\to\infty$, which is quite different from the fractal trees and Cayley trees as studied in \cite{JWZ}. This indicates that the network structure of hype-cubes are essentially different from that of fractal trees and Cayley trees constructed in \cite{JWZ}.

\subsection{Number of spanning trees}

In addition to eigentime identity, the eigenvalues of normalized Laplacian matrix of a connected network also determine the number of its spanning trees. Recall that a spanning tree of an undirected graph $\Gamma$ is a subgraph of $\Gamma$ that is a tree which includes all the vertices of $\Gamma$. In general, a graph may have several spanning trees, but a graph that is not connected will not contain a spanning tree. By \cite{Chu,CZ}, the number of spanning trees $N_{st}(\Gamma)$ for a connected network $\Gamma$ is
\be\label{nst}
N_{st}(\Gamma) = \frac{\prod_{i = 1}^{N}d_i\prod_{\lambda\neq 0}\lambda}{\sum_{i = 1}^{N}d_i},
\ee
where $\lambda$ runs over all the nonzero eigenvalues of the normalized Laplacian matrix.

Denote by $N_{st}(\Gamma_n)$ the number of spanning trees of the $n$-cube $\Gamma_n$.

\begin{prp}\label{st}
\be
N_{st}(\Gamma_n) = 2^{2^{n}-n-1}\prod_{k=1}^{n}k^{\binom{n}{k}}.
\ee
\end{prp}

\begin{proof}
By \eqref{nst} and Theorem \ref{res}, we have 
\[
N_{st}(\Gamma_n) = \frac{\prod d_i\prod_{\lambda\not = 0}\lambda}{\sum d_i} =\frac{n^{2^n}\prod_{k=1}^{n}(\frac{2k}{n})^{\binom{n}{k}}}{n2^n} = 2^{2^{n}-n-1}\prod_{k=1}^{n}k^{\binom{n}{k}}.
\]
\end{proof}

\section{Conclusions}

It is known that numerous structural and dynamical  properties of a networked system are encoded in eigenvalues and eigenvectors of its standard Laplacian matrix.
Compared to standard Laplacian matrices, the spectrum of normalized Laplacian matrices have received little attention. Recently, it was pointed out by \cite{JWZ} that
it is equally important to compute and analyze the normalized Laplacian spectrum. For example, the normalized Laplacian spectrum of a network is relevant in the topological aspects and 
random walk dynamics that is closely related to a large variety of other dynamical processes of the network.

Generally, it should be a difficult problem to determine explicitly the normalized Laplacian spectrum. However, we do believe that for vertex symmetric networks, especially for Cayley graph networks, the study of normalized Laplacian spectrum will be relatively easy, due to the symmetry of the network structure. 

In this paper, we have studied the eigenvalue problem of the normalized Laplacian matrices of the hype-cubes, a special kind of Cayley graph networks. We determined explicitly all the characteristic polynomials of the normalized Laplacian matrices of the hype-cubes by a recursive method. Particularly, these polynomials were factorized into products of monomials and the roots of these polynomials are elegantly distributed in the closed interval from 0 to 2. As an application of these results, we obtained explicitly the eigentime identity for random walks on these hype-cubes, which grows linearly with the network size.
This is in sharp contrast to fractal trees and Cayley trees as constructed in \cite{JWZ}. Since eigentime identity is an important quantity rooted in the inherent network topology, we conclude that the network structure of hype-cubes is essentially different from that constructed in \cite{JWZ}. Moreover, we derived the number of spanning trees of these hype-cubes through the normalized Laplacian spectrum.

{\bf Acknowledgments} \noindent This work was sponsored by the National Natural Science Foundation of China (NSFC) under Project No.61573119 and the Fundamental Research
Project of Shenzhen under Project Nos. JCYJ20170307151312215 and KQJSCX20180328165509766.

\end{document}